\documentclass[11pt,reqno]{amsart}
\usepackage{amsmath}
\usepackage{amssymb}
\usepackage{amsthm}
\usepackage[usenames]{color}
\usepackage{eepic,epic}

\newtheorem{thm}{Theorem}[section]

\newtheorem{lem}[thm]{Lemma}
\newtheorem{prop}[thm]{Proposition}

\theoremstyle{definition}

\theoremstyle{remark}
\newtheorem{rem}[thm]{Remark}
\numberwithin{equation}{section}

\newcommand{\R}{\mathbb{R}}

\begin{document}
\title[]
{Optimal convergence rate of nonrelativistic limit for the nonlinear pseudo-relativistic equations}
\author{Woocheol Choi}
\address{W. Choi, School of Mathematics, Korea Institute for Advanced Study, Seoul 02455, Korea}
\email{wchoi@kias.re.kr}

\author{Younghun Hong}
\address{Y. Hong, Department of mathematics, Yonsei University, Seoul 03722, Korea}
\email{younghun.hong@yonsei.ac.kr}

\author{Jinmyoung Seok}
\address{J. Seok, Department of Mathematics, Kyonggi University, Suwon 16227, Korea}
\email{jmseok@kgu.ac.kr}


\maketitle
\begin{abstract}
In this paper, we are concerned with the nonrelativistic limit of the following pseudo-relativistic equation with Hartree nonlinearity or power type nonlinearity
\[
\left(\sqrt{-\hbar^2c^2 \Delta +m^2c^4} - mc^2 \right) u + \mu u = \mathcal{N}(u),
\]
where $c$ denotes the speed of light. 
We prove that the ground states of this equation converges to the ground state of its nonrelativistic counterpart
\[
-\frac{\hbar^2}{2m}\Delta u + \mu u = \mathcal{N}(u)
\]
with an explicit convergence rate $1/c^2$ in arbitrary order as $c \to \infty$. Moreover, we show that this rate is optimal.

\end{abstract}

\section{introduction}

\subsection{Setup of the problem}
We consider the pseudo-relativistic equation
\begin{equation}\label{eq-i1}
i\hbar \partial_t\psi = (\sqrt{-\hbar^2c^2 \Delta + m^2 c^4}-mc^2)\psi -\mathcal{N}(\psi),
\end{equation}
where
$$\psi=\psi(t,x):\mathbb{R}\times\mathbb{R}^n\to\mathbb{C}$$
is the wave function, $n \geq 1$ is the space dimension, $\hbar$ is the reduced Planck constant, $c > 0$ denotes the speed of light and $m > 0$ represents the particle mass. The operator $\sqrt{-\hbar^2c^2 \Delta + m^2 c^4}$ is defined as the Fourier multiplier with symbol $\sqrt{\hbar^2c^2|\xi|^2 + m^2 c^4}$. The nonlinear term $\mathcal{N}(\psi)$ is assumed to be either of the power-type
$$\mathcal{N}(\psi)=|\psi|^{p-2}\psi$$
or of the Hartree-type
$$\mathcal{N}(\psi)=(|x|^{-1}*|\psi|^2)\psi.$$
The equation \eqref{eq-i1} is referred as the pseudo-relativistic nonlinear Schr\"odinger equation (NLS) (or nonlinear Hartree equation (NLH), respectively) when the nonlinearity is of the power- type (or the Hartree-type, respectively). Throughout the paper, we always assume that
$$n \geq 1\textup{ and }2< p <\tfrac{2n}{n-1}$$
$(\textrm{we set}~\tfrac{2n}{n-1}=\infty~ \textrm{when}~n=1)$ for the power-type nonlinearity, which makes $\mathcal{N}(u)$ be $H^{1/2}(\R^n)$ subcritical, while we assume that $n = 3$ when we refer $\mathcal{N}(u)$ to the Hartree nonlinearity.

This model is considered as a relativistic correction to the nonrelativistic counterpart, because by the Taylor series expansion of the symbol 
\begin{align*}
\sqrt{\hbar^2c^2|\xi|^2 + m^2 c^4}-mc^2&=mc^2\Big(\sqrt{1+\frac{\hbar^2|\xi|^2}{m^2c^2}}-1\Big)=\frac{\hbar^2|\xi|^2}{2m}-\frac{\hbar^4|\xi|^4}{8m^3c^2}+\cdots,
\end{align*}
the nonrelativistic equation
\begin{equation}\label{eq-i1''}
i\hbar \partial_t\psi =-\frac{\hbar^2}{2m}\Delta\psi -\mathcal{N}(\psi)
\end{equation}
formally approximates the equation \eqref{eq-i1} in the nonrelativistic regime
$$|\mathbf{p}|=\hbar|\xi|\ll mc.$$
Physically, the equation \eqref{eq-i1} is derived as a mean-field limit of a system of identical relativistic spin-$0$ bosons. It also provides a good approximation to the Bethe-Salpeter formalism \cite{CLM, Bethe-Salpeter}. In particular, the dynamics of boson stars is described by the three-dimensional pseudo-relativistic NLH, namely the \textit{boson star equation}. We refer \cite{ES, LT, LY} for rigorous derivation and \cite{KnPi, JLee} for the rate of convergence toward the mean-field limit. Because of the theoretical and experimental importance of the object, the dynamical study on the boson star equation also has been received a great attention past years \cite{FJL, FL1, FL2, HeLe, L1, LeMa, Pusateri}.

For consistency of the theory, an important question is whether the limit of a sequence of solutions to \eqref{eq-i1} as $\frac{mc}{\hbar}\to\infty$, called the \textit{nonrelativistic limit} if it exists, indeed solves the nonrelativistic equation. This kind of problem has been investigated in various physical settings \cite{BMS, BMS2, EsSe, MNO, MNO2, MN, MN2, MN3, MN4,Nakanishi, Ts}. A particular interest is on the limit along special solutions such as ground state solutions \cite{CS, EsSe,L2} and scattering solutions \cite{Nakanishi}.

In this paper, we are concerned with the limiting procedure of ground states for the pseudo-relativistic equation \eqref{eq-i1}. To introduce the ground states, we insert a standing wave ansatz $\psi(t,x) = e^{i\frac{\mu}{\hbar} t} u(x)$ in \eqref{eq-i1} and \eqref{eq-i1''} to have the time-independent pseudo-relativistic NLS (or NLH)
\begin{equation}\label{eq-elliptic}
(\sqrt{-\hbar^2c^2\Delta+m^2c^4}-mc^2)u +\mu u=\mathcal{N}(u),
\end{equation}
and the time-independent nonrelativistic NLS (or NLH)
\begin{equation}\label{eq-elliptic'}
-\frac{\hbar^2}{2m}\Delta u + \mu u = \mathcal{N}(u).
\end{equation}
From here, we normalize the equations and assume that $m=\frac{1}{2}$, $\mu=1$ and $c\to\infty$ for simplicity of the exposition. That is, we proceed with the pseudo-relativistic NLS (or NLH) of the form, 
\begin{equation}\label{pseudo-relativistic equation}
P_c(D)u=\mathcal{N}(u),
\end{equation}
where $P_c(D)$ is the pseudo-differential operator with symbol 
\begin{equation}\label{symbol}
P_c(\xi)=\sqrt{c^2|\xi|^2+\tfrac{1}{4}c^4}-\tfrac{c^2}{2}+1,
\end{equation}
and the nonrelativistic NLS (or NLH) of the form, 
\begin{equation}\label{nonrelativistic equation}
-\Delta u+u=\mathcal{N}(u).
\end{equation}
Indeed, the pseudo-relativistic equation is not scaling-invariant, so normalization may change the parameter $c$. However, since we take the nonrelativistic limit ($\frac{mc}{\hbar}\to\infty$), it will not affect our main result.

Now we state the notion of ground states and relevant results. In the nonrelativistic case, we associate the equation \eqref{nonrelativistic equation} with the functional
\begin{equation}\label{eq-E'}
I_\infty(u):= \frac{1}{2} \int_{\mathbb{R}^n}|\nabla u|^2dx-\frac{1}{p}\int_{\mathbb{R}^n} \mathcal{N}(u)\bar{u}dx, \qquad u \in H^1(\R^n),
\end{equation}
where $p=4$ for the Hartree nonlinearity. Then, we say that a critical point $u_\infty \in H^1(\R^n)$ of the functional $I_\infty$ is a \textit{ground state} to \eqref{eq-elliptic'} if 
$$I_\infty(u_\infty) = \min\Big\{ I_\infty(v)  ~|~ v \in H^1(\mathbb{R}^n) \setminus \{0\},\, I_\infty'(v) = 0\Big\},$$
where $I_\infty'$ denotes the Fr\'echet derivative of $I_\infty$. Existence of such a ground state is a classical result, see Berestycki-Lions \cite{BL} for NLS, and Lieb \cite {Lieb} for NLH for instance. Its uniqueness and nondegeneracy are established by Kwong \cite{Kw} for NLS, and Lenzmann \cite{L2} for NLH in three dimensions ($n=3$). Similarly, in the pseudo-relativistic case, we define the functional by
\begin{equation}\label{eq-E}
\begin{aligned}
I_c(u):&=\frac{1}{2}\int_{\mathbb{R}^n} P_c(D)u\overline{u}dx-\frac{1}{p}\int_{\mathbb{R}^n} \mathcal{N}(u)\bar{u}dx\\
&=\frac{1}{2}\cdot\frac{1}{(2\pi)^n} \int_{\mathbb{R}^n}(\sqrt{c^2|\xi|^2 +\tfrac{c^4}{4}}-\tfrac{c^2}{2})|{\hat u} (\xi)|^2d\xi-\frac{1}{p}\int_{\mathbb{R}^n} \mathcal{N}(u)\bar{u}dx,
\end{aligned}
\end{equation}
where $\hat{u}$ denotes the Fourier transform of $u$ and $p$ is again chosen to be $4$ for the Hartree nonlinearity. Note that $I_c$ is continuously differentiable on $H^{1/2}(\R^n)$ and its critical points are weak solutions of \eqref{eq-elliptic} and vice versa. We again say a critical point $u_c \in H^{1/2}(\mathbb{R}^n)$ of the functional $I_c$ is a \textit{ground state} to \eqref{eq-elliptic} if it satisfies
$$I_c(u_c) = \min\Big\{ I_c(v)  ~|~ v \in H^{1/2}(\mathbb{R}^n) \setminus \{0\},\, I_c'(v) = 0\Big\}.$$
Such a ground state is known to exist in \cite{LY} for NLH and Choi-Seok \cite{CS} for NLS. See also \cite{CN2,TWY,M}. 
The uniqueness of a ground state for NLH is shown by Lenzmann in \cite{L2} when the parameter $c$ is sufficiently large
by taking advantage of the nondegeneracy of the ground state of the nonrelativistic NLH, which is proved by himself in \cite{L2} as mentioned earlier. 
One also can deduce the uniqueness of a ground state for NLS at least for large $c$ from the nondegeneracy of a ground state of nonrelativistic NLS.

By using the strong maximum principle, it can be shown that both pseudo-relativistic and nonrelativistic ground states have only one sign so we may assume they are strictly positive. Therefore, with abuse of notation, we omit the absolute value sign in the nonlinear term:
\begin{equation}\label{nonlinearity'}
\mathcal{N}(u):=\left\{\begin{aligned}
&\Big(\frac{1}{|x|}*u^2\Big)u&&\textup{(Hartree nonlinearity)}\\
&u^{p-1}&&\textup{(power-type nonlinearity)}.
\end{aligned}
\right.
\end{equation}

Concerning the nonrelativistic limit, by Lenzmann \cite{L2} for NLH and  by Choi-Seok \cite{CS} for NLS, it is proved that a family of ground states of \eqref{eq-elliptic} converges to the ground state of \eqref{eq-elliptic'}. The exact statement is as the following. 
\begin{thm}[$H^1$ convergence \cite{CS,L2}]
Suppose that $n=3$ for NLH and $2<p<\frac{2n}{n-1},\, n \geq 1$ for NLS. Then, for each $c\geq 1$, a positive radially symmetric ground state solution $u_c$ of \eqref{pseudo-relativistic equation} belongs to $H^1(\R^n)$. Moreover, the sequence $\{u_c\}_{c\ge 1}$ converges to the unique positive radially symmetric ground state solution $u_\infty$ of \eqref{nonrelativistic equation} as $c \to \infty$ in the sense that
\begin{equation*}
\lim_{c \rightarrow \infty} \| u_c - u_{\infty}\|_{H^1(\mathbb{R}^n)} = 0.
\end{equation*}
\end{thm}
\begin{rem}
In fact, the existence of a ground state and its $H^1$ convergence for NLS are proved in \cite{CS} only when $n \geq 2$. 
One can however obtain ground states for $n =1$ by arguing similarly to \cite{CS} but adopting the concentration compactness argument \cite{L} instead of
using compact embedding $H^{1/2}_r(\R^n) \hookrightarrow L^p(\R^n)$, which is not true when $n = 1$.
Here, $H^{1/2}_r(\R^n)$ denotes the set of all radial functions in $H^{1/2}(\R^n)$.
\end{rem}

\subsection{Statement of the main result}
In this paper, we further exploit the nonrelativistic limit of a pseudo-relativistic ground state in the earlier works by Lenzmann \cite{L2} and Choi and Seok \cite{CS}. 
Precisely, we find the explicit convergence rate of ground states, in terms of $c > 1$, in an arbitrarily high order Sobolev norm.
In addition, we show that this rate is optimal.

Namely, the main result of this paper is as follows. 

\begin{thm}[Nonrelativistic limit for NLS/NLH]\label{thm-1-2}
Suppose that $n=3$ for NLH and $n \geq 1$,\, $p \in (2,\,\frac{2n}{n-1})\cap \mathbb{N}$ for NLS. Let $u_c$ be a positive radially symmetric ground state of \eqref{pseudo-relativistic equation}, and let $u_\infty$ be the unique positive radially symmetric ground state of \eqref{nonrelativistic equation}. Then $u_c$ belongs to $H^s(\R^n)$ for every $s \geq 1$, and moreover there exist positive constants $A$ and $B$ such that
\begin{equation}
\frac{A}{c^2} \leq \|u_c-u_\infty\|_{H^s(\R^n)} \leq \frac{B}{c^2},
\end{equation}
where  $A$ and $B$ depend only on $s$ for the Hartree nonlinearity and $s,\, n,\, p$ for the power-type nonlinearity. 
\end{thm}

\begin{rem}\mbox{~}
\begin{enumerate}
\item Scaling back to the original equations \eqref{eq-elliptic} and \eqref{eq-elliptic'}, we obtain the optimal $O((\frac{mc}{\hbar})^{-2})$ rate of convergence to the non-relativistic ground state.
\item The power-type nonlinearity $u^{p-1}$ admits integer powers in the following cases:
\begin{equation}\label{integer power assumption}
\left\{\begin{aligned}
& n = 1 \text{ and } p\geq 3 \text{ is integer}, \\
& n = 2 \text{ and } p=3.
\end{aligned}
\right.
\end{equation}
\item In the statement of our main theorem, the NLS with fractional power nonlinearity is excluded, because in this case, even the limit $u_\infty$ is not known to be contained in a higher order Sobolev space $H^s$ due to the lack of differentiability of $u^{p-1}$. We report without a proof that when $p$ is fractional, it is possible to obtain the $H^2(\R^n)$ convergence with a weaker $O(1/c)$ rate by a simple modification of the proof of Theorem \ref{thm-1-2}.
\end{enumerate}
\end{rem}

To the best of authors' knowledge, Theorem \ref{thm-1-2} seems the first result showing the explicit convergence rates in a high Sobolev norm for the nonrelativistic limit of ground states. For achieving this, more refined information on the ground state of the limit equation \eqref{nonrelativistic equation} is required. 
Namely,  we exploit carefully the non-degeneracy of the limit solution $u_{\infty}$ of problem \eqref{nonrelativistic equation}.
We recall that  a solution $u_\infty \in H^1_r(\R^n)$ of \eqref{eq-elliptic'} is said to be non-degenerate in radial class if the linearized equation of \eqref{eq-elliptic'} at $u_\infty$,
\[
L_{u_\infty}[\phi] := -\Delta\phi + \phi -\mathcal{N}'(u_\infty)\phi = 0,
\]
only admits the trivial solution $\phi = 0$ on the class $H^1_r(\R^n)$.
This says that the second eigenvalue of the linearized operator $L_{u_\infty}$ is strictly positive so there exists a small constant $c > 0$ satisfying
\[
c\|\phi\|_{H^1(\R^n)}^2  \leq \int_{\R^n}|\nabla\phi|^2 + \phi^2 -\mathcal{N}'(u_\infty)\phi^2\,dx
\]
for every $\phi \in H^1_r(\R^n)$ orthogonal to $u_\infty$ in $H^1$. We will see that this inequality provides with a crucial estimate for proving our main theorem.


Our proof is rather systematic, and it does not depend highly on the structure of the equation \eqref{pseudo-relativistic equation}. Instead, it relies only on sub-criticality of the equations (for a high Sobolev norm convergence) and non-degeneracy of the ground state solution to the limiting equation (for the rate of convergence). We believe that it can be easily adapted to other physical models.

\subsection{Organization of the paper}
In Section \ref{sec-2}, as a preliminary, we prove the iterated nonlinear estimates (Proposition \ref{iterated Hartree nonlinear estimates} and \ref{iterated integer-power nonlinear estimates}) to handle the $H^{1/2}$-subcritical nonlinearities. In Section  \ref{sec-3}, we apply the iterated estimates to obtain a high Sobolev norm uniform bound on the pseudo-relativistic ground state $u_c$ (Proposition \ref{Uniform Sobolev norm bounds for the pseudo-relativistic equations}). In Section \ref{sec-10}, we prove the sharp upper and lower estimates for the nonrelativistic limit in the space $H^1 (\mathbb{R}^n)$. The convergence norm is improved in Section \ref{sec-4} to the Sobolev space of any order. 
In the appendices, for the reader's convenience, we summarize the basic properties of the Lorentz spaces and the fractional Leibniz rule.

\

\noindent \emph{Notation.} We shall use the notation $A\lesssim B$ (resp., $A\gtrsim B$) when $A \leq CB$ (resp., $A \geq CB$) holds with a constant $C>1$ not depending on the parameter $c >1$. Also, we shall denote $A\sim B$ if both $A \lesssim B$ and $A \gtrsim B$ hold.

\subsection{Acknowledgement} 
This research of the first and third authors was supported by the POSCO TJ Park Science Fellowship.
This research of the second author was supported by Basic Science Research Program through the National Research Foundation of Korea(NRF) funded by the Ministry of Education (NRF-2015R1A5A1009350).
This research of the third author was supported by Basic Science Research Program through the National Research Foundation of Korea(NRF) funded by the Ministry of Education (NRF-2014R1A1A2054805).

\section{Iterated nonlinear estimates}\label{sec-2}

We prove the iterated multi-linear estimates in the Sobolev space $H^s$, where $H^s$ is the Sobolev space with the norm 
$$\|u\|_{H^s(\mathbb{R}^n)}: =\Big\{\frac{1}{(2\pi)^n}\int_{\mathbb{R}^n} (1+|\xi|^2)^{s} |\hat{u}(\xi)|^2 d \xi\Big\}^{1/2}.$$
For the Hartree nonlinearity, we show the following iterated trilinear estimates.

\begin{prop}[Iterated estimate for the Hartree nonlinearity]\label{iterated Hartree nonlinear estimates}
There exists a strictly increasing sequence $\{s_0\}_{k=1}^\infty$, with $s_0=\frac{1}{2}$ and $s_k\to\infty$, such that
\begin{equation}\label{eq: iterated Hartree nonlinear estimates}
\Big\|\Big(\frac{1}{|x|}*(v_1 v_2)\Big)v_3\Big\|_{H^{s_{k+1}-1}(\mathbb{R}^3)}\lesssim \prod_{j=1}^3\|v_j\|_{H^{s_k}(\mathbb{R}^3)}.
\end{equation}
\end{prop}

The proposition is immediately deduced from the elementary trilinear estimate with $s_{k}:=k+\frac{1}{2}$.

\begin{lem}[Trilinear estimate for the Hartree nonlinearity]\label{Hartree nonlinear estimates}
$$\Big\|\Big(\frac{1}{|x|}*(v_1 v_2)\Big)v_3\Big\|_{H^{s}(\mathbb{R}^3)}\lesssim\prod_{j=1}^3\|v_j\|_{H^s(\mathbb{R}^3)},\quad \forall s\geq\tfrac{1}{2}.$$
\end{lem}

\begin{proof}
By the fractional Leibniz rule, we get
\begin{align*}
\Big\|\Big(\frac{1}{|x|}*(v_1 v_2)\Big)v_3\Big\|_{H^{s}(\mathbb{R}^3)}&=\Big\|\langle\nabla\rangle^s\Big(\frac{1}{|x|}*(v_1 v_2)\Big)v_3\Big\|_{L^2(\mathbb{R}^3)}\\
&\lesssim \Big\|\frac{1}{|x|}*(v_1v_2)\Big\|_{L^\infty(\mathbb{R}^3)}\|v_3\|_{H^s(\mathbb{R}^3)}\\
&\quad+ \Big\||\nabla|^s\Big(\frac{1}{|x|}*(v_1v_2)\Big)\Big\|_{L^{6}(\mathbb{R}^3)}\|v_3\|_{L^3(\mathbb{R}^3)}.
\end{align*}
By the H\"older and Sobolev inequalities in the Lorentz spaces, 
\begin{align*}
\Big\|\frac{1}{|x|}*(v_1v_2)\Big\|_{L^\infty(\mathbb{R}^3)}\|v_3\|_{H^s(\mathbb{R}^3)}&\lesssim\|v_1v_2\|_{L^{\frac{3}{2},1}(\mathbb{R}^3)}\|v_3\|_{H^s(\mathbb{R}^3)}\\
&\lesssim \|v_1\|_{L^{3,2}(\mathbb{R}^3)}\|v_2\|_{L^{3,2}(\mathbb{R}^3)}\|v_3\|_{H^s(\mathbb{R}^3)}\\
&\lesssim\prod_{j=1}^3\|v_j\|_{H^s(\mathbb{R}^3)}
\end{align*}
and
\begin{align*}
&\Big\||\nabla|^s\Big(\frac{1}{|x|}*(v_1v_2)\Big)\Big\|_{L^{6}(\mathbb{R}^3)}\|v_3\|_{L^3(\mathbb{R}^3)}\\
&=\Big\|\frac{1}{|x|}*|\nabla|^s(v_1v_2)\Big\|_{L^{6}(\mathbb{R}^3)}\|v_3\|_{L^3(\mathbb{R}^3)}\\
&\lesssim \||\nabla|^s(v_1v_2)\|_{L^{\frac{6}{5}}(\mathbb{R}^3)}\|v_3\|_{L^3(\mathbb{R}^3)}\\
&\lesssim \Big(\|v_1\|_{H^s(\mathbb{R}^3)}\|v_2\|_{L^3(\mathbb{R}^3)}+\|v_1\|_{L^3(\mathbb{R}^3)}\|v_2\|_{H^s(\mathbb{R}^3)}\Big)\|v_3\|_{L^3(\mathbb{R}^3)}\\
&\lesssim \prod_{j=1}^3\|v_j\|_{H^s(\mathbb{R}^3)},
\end{align*}
where in the first inequality, we used the Hardy-Littlewood-Sobolev  inequality. Therefore, we obtain the trilinear estimate in the lemma.
\end{proof}

We prove an analogous iterated estimate for the integer-power nonlinearity. 

\begin{prop}[Iterated estimate for the integer-power nonlinearity]\label{iterated integer-power nonlinear estimates}
Suppose that $p\in(2,\frac{2n}{n-1})$ is an integer. Then, there exists a strictly increasing sequence $\{s_0\}_{k=1}^\infty$, with $s_0=\frac{1}{2}$ and $s_k\to\infty$, such that
\begin{equation}\label{eq: iterated integer-power nonlinear estimates}
\Big\|\prod_{j=1}^{p-1}v_j\Big\|_{H^{s_{k+1}-1}(\mathbb{R}^n)}\lesssim \prod_{j=1}^{p-1}\|v_j\|_{H^{s_k}(\mathbb{R}^n)}.
\end{equation}
\end{prop}

As in the proof of Proposition \ref{iterated Hartree nonlinear estimates}, we show the proposition using the multi-linear estimates.

\begin{lem}[Multi-linear estimates for the power-type nonlinearity]\label{integer-power nonlinear estimates}
Suppose that $p\in(2,\frac{2n}{n-1})$ is an integer.\\
$(i)$ If $\frac{1}{2}\leq s<\frac{n}{2}$ $(\Rightarrow n\geq 2)$, then
$$\Big\|\prod_{j=1}^{p-1}v_j\Big\|_{H^{(p-1)s-\frac{n(p-2)}{2}}(\mathbb{R}^n)}\lesssim \prod_{j=1}^{p-1}\|v_j\|_{H^s(\mathbb{R}^n)}.$$
$(ii)$ For any small $\epsilon>0$,
$$\Big\|\prod_{j=1}^{p-1}v_j\Big\|_{H^{\frac{n-\epsilon}{2}}(\mathbb{R}^n)}\lesssim \prod_{j=1}^{p-1}\|v_j\|_{H^{\frac{n}{2}}(\mathbb{R}^n)}.$$
$(iii)$ If $s>\frac{n}{2}$, then 
$$\Big\|\prod_{j=1}^{p-1}v_j\Big\|_{H^{s}(\mathbb{R}^n)}\lesssim \prod_{j=1}^{p-1}\|v_j\|_{H^s(\mathbb{R}^n)}.$$
\end{lem}

\begin{proof}
We prove the lemma using the Sobolev and H\"older inequalities and the fractional Leibniz rule as in the proof of Lemma \ref{Hartree nonlinear estimates}. For $(i)$, when $(p-1)s-\frac{n(p-2)}{2}\leq 0$, we obtain
\begin{align*}
\Big\|\prod_{j=1}^{p-1}v_j\Big\|_{H^{(p-1)s-\frac{n(p-2)}{2}}(\mathbb{R}^n)}&\lesssim \Big\|\prod_{j=1}^{p-1}v_j\Big\|_{L^{\frac{2n}{(p-1)(n-2s)}}(\mathbb{R}^n)}\\
&\lesssim \prod_{j=1}^{p-1}\|v_j\|_{L^{\frac{2n}{n-2s}}(\mathbb{R}^n)}\\
&\lesssim \prod_{j=1}^{p-1}\|v_j\|_{H^s(\mathbb{R}^n)}.
\end{align*}
When $(p-1)s-\frac{n(p-2)}{2}\leq 0$, by fractional Leibniz rule, 
\begin{align*}
\Big\|\prod_{j=1}^{p-1}v_j\Big\|_{H^{(p-1)s-\frac{n(p-2)}{2}}(\mathbb{R}^n)}&\lesssim \sum_{\ell=1}^{k-1}\prod_{j=1, j\neq \ell}^{p-1}\|v_j\|_{L^{\frac{2n}{n-2s}}(\mathbb{R}^n)}\|v_{\ell}\|_{W^{(p-1)s-\frac{n(p-2)}{2},\frac{2n}{n-(n-2s)(p-2)}}(\mathbb{R}^n)}\\
&\lesssim \sum_{\ell=1}^{k-1}\prod_{j=1}^{p-1}\|v_j\|_{H^s(\mathbb{R}^n)}\lesssim\prod_{j=1}^{p-1}\|v_j\|_{H^s(\mathbb{R}^n)}.
\end{align*}
Here, we used that
\begin{align*}
n-(n-2s)(p-2)&>n-(n-2s)\cdot(\tfrac{2n}{n-1}-2)=n-(n-2s)\cdot(\tfrac{2}{n-1})\quad\textup{(by $p<\tfrac{2n}{n-1}$)}\\
&\geq n-(n-1)\cdot\tfrac{2}{n-1}=n-2\geq0\quad\textup{(by $n\geq 2$)}.
\end{align*}
For $(ii)$, by the fractional Leibniz rule and the Sobolev embedding $L^q(\mathbb{R}^n)\hookrightarrow H^{\frac{n}{2}}(\mathbb{R}^n)$ for all $2\leq q<\infty$,
\begin{align*}
\Big\|\prod_{j=1}^{p-1}v_j\Big\|_{H^{\frac{n-\epsilon}{2}}(\mathbb{R}^n)}&=\Big\|\langle\nabla\rangle^{\frac{n-\epsilon}{2}}\prod_{j=1}^{p-1}v_j\Big\|_{L^2(\mathbb{R}^n)}\\
&\lesssim \sum_{\ell=1}^{k-1}\prod_{j=1, j\neq \ell}^{p-1}\|v_j\|_{L^{\frac{2n(p-2)}{\epsilon}}(\mathbb{R}^n)}\|v_{\ell}\|_{W^{\frac{n-\epsilon}{2},\frac{2n}{n-\epsilon}}(\mathbb{R}^n)}\\
&\lesssim \|u\|_{H^{\frac{n}{2}}(\mathbb{R}^n)}^{p-1}.
\end{align*}
For $(iii)$, we just use the Sobolev embedding $L^\infty(\mathbb{R}^n)\hookrightarrow H^s(\mathbb{R}^n)$ to prove
\begin{align*}
\Big\|\prod_{j=1}^{p-1}v_j\Big\|_{H^s(\mathbb{R}^n)}&=\Big\|\langle\nabla\rangle^s\prod_{j=1}^{p-1}v_j\Big\|_{L^2(\mathbb{R}^n)}\\
&\lesssim \sum_{\ell=1}^{k-1}\prod_{j=1, j\neq \ell}^{p-1}\|v_j\|_{L^\infty(\mathbb{R}^n)}\|v_{\ell}\|_{H^s(\mathbb{R}^n)}\\
&\lesssim \|u\|_{H^s(\mathbb{R}^n)}^{p-1}.
\end{align*}
\end{proof}

\begin{proof}[Proof of Proposition \ref{iterated integer-power nonlinear estimates}]
We define $s_1,\cdots, s_K$ by the recursive formula
$$s_{k+1}:=1-\tfrac{n(p-2)}{2}+(p-1)s_k,$$
equivalently $s_k=\frac{n}{2}-\frac{1}{p-2}+(\frac{1}{p-2}-\frac{n-1}{2})\cdot (p-1)^k$, where $K$ is the first integer making $s_K>\frac{n}{2}$. Note that since $p$ is assumed to be contained in the $H^{1/2}$-subcritical range $(2,\frac{2n}{n-1})$, the coefficient $\frac{1}{p-2}-\frac{n-1}{2}$ is strictly positive. Hence, such a $K$ exists. Here, if $s_{K-1}=\frac{n}{2}$, then replacing $s_{K-1}$ and $S_K$ by slightly less numbers, we may assume that $s_{K-1}<\frac{n}{2}<s_K$. For $k\geq K$, we define
$$s_{k+1}:=s_k+1.$$
By construction, $s_k\to\infty$ as $k\to\infty$. The inequality \eqref{eq: iterated integer-power nonlinear estimates} can be proved by Lemma \ref{integer-power nonlinear estimates} $(i)$ when $j\leq K-1$ $(\Rightarrow s_j<\frac{n}{2})$, while it can be proved by Lemma \ref{integer-power nonlinear estimates} $(ii)$ or $(iii)$ when $j\geq K$ $(\Rightarrow s_j\geq\frac{n}{2})$.
\end{proof}

\section{A uniform higher Sobolev norm bound on the ground state}\label{sec-3}

We denote by $u_c$ a ground state solution to the pseudo-relativistic NLS (or NLH), 
$$P_c(D)u=\mathcal{N}(u),$$
where $P_c(D)=\sqrt{-c^2\Delta+\tfrac{1}{4}c^4}-\tfrac{c^2}{2}+1$ and the nonlinearity is given by \eqref{nonlinearity'}. We recall that the ground state $u_c$ is uniformly bounded in $H^{1/2}$. 
See Lemma 3 in Lenzmann \cite{L2} for NLH and Lemma 5.2 in Choi-Seok \cite{CS} for NLS with $n \geq 2$. Arguing similarly to \cite{CS}, this still holds for NLS with $n = 1$.
\begin{lem}[Uniform $H^{1/2}$-bound]\label{uniform H 1/2 bound}
$$\sup_{c\geq 1}\|u_c\|_{H^{1/2}(\mathbb{R}^n)}<\infty.$$
\end{lem}

The purpose of this section is to show that the low Sobolev norm uniform bound in Lemma \ref{uniform H 1/2 bound} can be upgraded to a higher Sobolev norm uniform bound.
\begin{prop}[Uniform high Sobolev norm bound]\label{Uniform Sobolev norm bounds for the pseudo-relativistic equations}
$$\sup_{c\geq 1}\|u_c\|_{H^s(\mathbb{R}^n)}<\infty,\quad\forall s\geq\tfrac{1}{2}.$$
\end{prop}

Proposition \ref{Uniform Sobolev norm bounds for the pseudo-relativistic equations} improves the $H^1$-uniform bound for NLH in \cite{L2} and that for NLS in \cite{CS}.
Although we believe that this can be obtained by further developing the method employed in \cite{L2,CS}, we shall not follow this approach.
Instead, we pay attention to a key observation that the symbol of the linear part of the equation has a uniform lower bound (Lemma \ref{symbol bound}), and therefore we may apply a rather general elliptic regularity theorem (Theorem \ref{elliptic regularity}) for the desired high Sobolev norm bound. We found that this approach is algebraically simpler in particular when $s$ is large.

\subsection{Uniform bound on the symbol}
The pseudo-relativistic NLS (or NLH) contains the differential operator $P_c(D)$ on the left hand side, whose symbol is 
$$P_c(\xi)=\sqrt{c^2|\xi|^2+\tfrac{1}{4}c^4}-\tfrac{c^2}{2}+1.$$
This symbol obeys a uniform lower bound of order 1.
\begin{lem}[Uniform lower bound on the symbol $P_c$]\label{symbol bound}
There exists a constant $K>0$ such that
$$P_c(\xi)\geq K\langle\xi\rangle,\quad\forall\xi\in\mathbb{R}^n.$$
\end{lem}

\begin{proof}
Factoring out $\frac{c^4}{4}$ from the square root term in the symbol, we write
$$P_c(\xi)=\frac{c^2}{2}\Big(\sqrt{1+\Big|\frac{2\xi}{c}\Big|^2}-1\Big)+1=\frac{c^2}{2}\cdot f\big(|\tfrac{2\xi}{c}|^2\big)+1,$$
where
$$f(t):=\sqrt{1+t}-1.$$
By the first degree Taylor expansion, if $0\leq t\leq 4$, then
\begin{align*}
f(t):&=\sqrt{1+t}-1=f(0)+f'(t_*) t\quad\textup{for some }t_*\in[0,4]\\
&=\frac{t}{2\sqrt{1+t_*}}\geq\frac{t}{2\sqrt{5}}.
\end{align*}
Hence, if $|\xi|\leq c$, then
\begin{align*}
P_c(\xi)&=\frac{c^2}{2}\Big(\sqrt{1+\Big|\frac{2\xi}{c}\Big|^2}-1\Big)+\mu=\frac{c^2}{2}f\big(|\tfrac{2\xi}{c}|^2\big)+\mu\\
&\geq \frac{c^2}{\sqrt{5}}\Big|\frac{2\xi}{c}\Big|^2+1\geq \frac{4|\xi|^2}{\sqrt{5}}+1\gtrsim\langle\xi\rangle.
\end{align*}
On the other hand, if $|\xi|\geq c$, then we have
\begin{align*}
P_c(\xi)&=c|\xi|\sqrt{1+\Big(\frac{c}{2|\xi|}\Big)^2}-\frac{c^2}{2}+1\geq c|\xi|-\frac{c^2}{2}+1\\
&=c|\xi|\Big(1-\frac{c}{2|\xi|}\Big)+1\geq\frac{c|\xi|}{2}+1\geq\frac{|\xi|}{2}+1\geq\frac{\langle\xi\rangle}{2}.
\end{align*}
\end{proof}

\begin{rem}
By Lemma \ref{symbol bound}, one may formally say that $P_c(D)^{-1}\lesssim\langle\nabla\rangle^{-1}$, or $\frac{\langle\nabla\rangle}{P_c(D)}\lesssim 1$. However, rigorously, these formal inequalities make sense only as an operator on the $L^2$-based Sobolev space $H^s$. Indeed, the symbol $\frac{\langle\xi\rangle}{P_c(\xi)}$ is not a Mikhlin multiplier, because the derivative of the symbol $\frac{\langle\xi\rangle}{P_c(\xi)}$ is not bounded uniformly in $c\geq1$. Thus, it is not necessary that $\frac{\langle\nabla\rangle}{P_c(D)}$ is bounded on $L^p$ for $p\geq 2$.
\end{rem}

\subsection{Elliptic regularity}
Now we consider a general class of nonlinear elliptic equations of the form
\begin{equation}\label{general elliptic PDE}
P(D)u=\mathcal{N}(u),
\end{equation}
where $u:\mathbb{R}^n\to\mathbb{R}$, the differential operator $P(D)$ is the multiplier operator with the symbol $P:\mathbb{R}^n\to\mathbb{R}$, i.e.,
$$\widehat{P(D)u}(\xi)=P(\xi)\widehat{u}(\xi),$$
and the nonlinear term is given by \eqref{nonlinearity'}.


The following theorem asserts that an $H^{1/2}$-solution to \eqref{general elliptic PDE} can be upgraded to an $H^s$-solution for any $s\geq\frac{1}{2}$, and moreover its higher Sobolev norm bound depends only on the order and the $H^{1/2}$ Sobolev norm bound.

\begin{thm}[Elliptic regularity]\label{elliptic regularity}
Suppose that
\begin{equation}\label{general symbol bound}
P(\xi)\geq K\langle\xi\rangle.
\end{equation}
Let $u$ be a solution of \eqref{general elliptic PDE} satisfying
$$\|u\|_{H^{1/2}(\mathbb{R}^n)}\leq B<\infty.$$
Then, 
\begin{equation}\label{eq: elliptic regularity}
\|u\|_{H^s(\mathbb{R}^n)}\leq C_{s},\quad\forall s\geq\tfrac{1}{2},
\end{equation}
where the constant $C_s$ depend only on $s, n, B, K$ (and $p$).
\end{thm}

\begin{proof}
By interpolation, it suffices to show that there exists a sequence $\{s_k\}_{k=0}^\infty$, with $s_0=\frac{1}{2}$ and $s_k\to\infty$, such that \eqref{eq: elliptic regularity} holds for $s=s_k$. Indeed, if the sequence $\{s_k\}_{k=0}^\infty$ is taken from Proposition \ref{iterated Hartree nonlinear estimates} (resp., \ref{iterated integer-power nonlinear estimates}) 
for the Hartree nonlinearity (resp., the integer-power nonlinearity), then
\begin{align*}
\|u\|_{H^{s_{k+1}}(\mathbb{R}^n)}&=\|P(D)^{-1}\mathcal{N}(u)\|_{H^{s_{k+1}}(\mathbb{R}^n)}\quad\textup{(by \eqref{general elliptic PDE})}\\
&\lesssim_K \|\mathcal{N}(u)\|_{H^{s_{k+1}-1}(\mathbb{R}^n)}\quad\textup{(by \eqref{general symbol bound})}\\
&=\Biggl\{\begin{aligned}
&\Big\|\Big(\frac{1}{|x|}*u^2\Big)u\Big\|_{H^{s_{k+1}-1}(\mathbb{R}^3)}\lesssim\|u\|_{H^{s_k}(\mathbb{R}^3)}^3&&\textup{(by Proposition \ref{iterated Hartree nonlinear estimates})}\\
&\|u^{p-1}\|_{H^{s_{k+1}-1}(\mathbb{R}^n)}\lesssim\|u\|_{H^{s_k}(\mathbb{R}^n)}^{p-1}&&\textup{(by Proposition \ref{iterated integer-power nonlinear estimates})},
\end{aligned}
\Biggr.
\end{align*}
depending on the nonlinearity. Iterating this inequality from $s_0=\frac{1}{2}$, we complete the proof.
\end{proof}

\subsection{Proof of Proposition \ref{Uniform Sobolev norm bounds for the pseudo-relativistic equations}}
Since the set of symbols $\{P_c\}_{c\geq1}$ enjoys the uniform lower bound as in Lemma \ref{symbol bound}, we may apply Theorem \ref{elliptic regularity} to obtain the desired uniform high Sobolev norm bound.

\section{The sharp convergence rate}\label{sec-10}

We denote by $u_\infty \in H^1(\R^n)$ the positive radial ground state solution to
\begin{equation}\label{eq-s-3}
-\Delta u + u = \mathcal{N}(u),
\end{equation}
and let $u_c \in H^{1}(\R^n)$ be a family of positive radial ground state solutions to
\begin{equation}\label{eq-s-4}
P_c(D) u=\Big(\sqrt{-c^2\Delta +\tfrac{c^4}{4}}-\tfrac{c^2}{2}\Big)u_c + u_c = \mathcal{N}(u),
\end{equation}
such that $u_c$ converges $u_\infty$ in $H^1 (\mathbb{R}^n)$ as $c \rightarrow \infty$.

In this section, we obtain the optimal $O(\frac{1}{c^2})$-rate convergence of $u_c$ to the nonrelativistic limit $u_\infty$ in the $H^1 (\mathbb{R}^n)$ space. 

\begin{prop}[Rate of convergence]\label{prop: rate of convergence}
Assume that the nonlinearity $\mathcal{N}(u)$ is either of the Hartree-type or of the power-type $|u|^{p-2}u$ with $p \in (2,\, \frac{n-1}{2n}) \cap \mathbb{N}$. Then, 
\begin{equation}
\| u_c - u_\infty\|_{H^1(\mathbb{R}^n)}=O\left(\frac{1}{c^2}\right).
\end{equation}
\end{prop}

The key ingredient of the proof of Proposition \ref{prop: rate of convergence} is the non-degeneracy of the linearized operator
$$\mathcal{L}:=-\Delta +1 -\mathcal{N}'(u_\infty)$$
for radial functions. Here the mapping $\mathcal{N}' (u_{\infty}): H^1(\mathbb{R}^n) \mapsto H^{-1} (\mathbb{R}^n)$ is given by
\begin{align*}
\mathcal{N}'(u_\infty)v =\left\{\begin{aligned}
&\Big(\frac{1}{|x|}*u_\infty^2\Big)v+2u_\infty\Big(\frac{1}{|x|}*(u_\infty v)\Big)&&\textup{(for NLH)}\\
&(p-1)u_\infty^{p-2}\,v &&\textup{(for NLS)}.
\end{aligned}
\right.
\end{align*}

\begin{lem}[Non-degeneracy \cite{Kw, L2}]\label{lem: non-degeneracy}
There exists a small $d > 0$ such that 
\[
\int_{\R^n}|\nabla v|^2 +v^2 -\mathcal{N}'(u_\infty)v^2\,dx\geq d\|v\|_{H^1(\mathbb{R}^n)}^2
\]
for every radial $v \in H^1(\R^n)$ such that $\langle u, v\rangle_{H^1} = 0$.
\end{lem}

\begin{proof}[Proof of Proposition \ref{prop: rate of convergence}]
For the difference $w:=u_c-u_\infty$, we do algebraic manipulations with the equations for $u_c$ and $u_\infty$,
\begin{align*}
\|w\|_{H^1(\mathbb{R}^n)}^2&=\int_{\mathbb{R}^n}(-\Delta+1)w\cdot w\ dx\\
&=\int_{\mathbb{R}^n}(-\Delta+1)(u_c-u_\infty)\cdot w\ dx\\
&=\int_{\mathbb{R}^n}\Big\{(-\Delta+1-P_c(D))u_c+P_c(D)u_c-(-\Delta+1)u_\infty\Big\}\cdot w\ dx\\
&=\int_{\mathbb{R}^n}\Big\{(-\Delta+1-P_c(D))u_c+\mathcal{N}(u_c)-\mathcal{N}(u_\infty)\Big\}\cdot w\ dx\\
&=\int_{\mathbb{R}^n}\Big\{(-\Delta+1-P_c(D))u_c+\mathcal{N}(u_\infty+w)-\mathcal{N}(u_\infty)\Big\}\cdot w\ dx.
\end{align*}
Taking out the first degree Taylor approximation from $\mathcal{N}(u_\infty+w)-\mathcal{N}(u_\infty)$, that is, $\mathcal{N}'(u_\infty)w$, we write
\begin{align*}
\|w\|_{H^1(\mathbb{R}^n)}^2&=\int_{\mathbb{R}^n}(-\Delta+1-P_c(D))u_c\cdot w\ dx+\int_{\mathbb{R}^n}\mathcal{N}'(u_\infty)w\cdot w\ dx\\
&\quad+\int_{\mathbb{R}^n}\Big\{\mathcal{N}(u_\infty+w)-\mathcal{N}(u_\infty)-\mathcal{N}'(u_\infty)w\Big\}\cdot w\ dx.
\end{align*}
Then, moving the second integral to the left, 
\begin{equation}\label{Lww}
\begin{aligned}
\int_{\mathbb{R}^n}\mathcal{L}w\cdot w\ dx&=\int_{\mathbb{R}^n}(-\Delta+1-P_c(D))u_c\cdot w\ dx\\
&\quad+\int_{\mathbb{R}^n}\Big\{\mathcal{N}(u_\infty+w)-\mathcal{N}(u_\infty)-\mathcal{N}'(u_\infty)w\Big\}\cdot w\ dx,
\end{aligned}
\end{equation}
where
\begin{align*}
\mathcal{L}v&=(-\Delta+1-\mathcal{N}'(u_\infty))v\\
&=\left\{\begin{aligned}
&(-\Delta+1)v-\Big(\frac{1}{|x|}*u_\infty^2\Big)v-2u_\infty\Big(\frac{1}{|x|}*(u_\infty v)\Big)&&\textup{(Hartree nonlinearity)}\\
&(-\Delta+1-(p-1)u_\infty^{p-2})v &&\textup{(power-type nonlinearity)},
\end{aligned}
\right.
\end{align*}
is the linearized operator around the ground state $u_\infty$. For both nonlinearities, 
\begin{equation}\label{algebraic property of L}
\mathcal{L}u_\infty=(-\Delta+1)u_\infty-(p-1)\mathcal{N}(u_\infty)=-(p-2)(-\Delta+1)u_\infty,
\end{equation}
where $p=4$  for the Hartree nonlinearity.

For the first integral on the right hand side of \eqref{Lww}, by Cauchy-Schwartz, 
$$\int_{\mathbb{R}^n}(-\Delta+1-P_c(D))u_c\cdot w\ dx\leq \big\|(-\Delta+1-P_c(D))u_c\big\|_{H^{-1}(\mathbb{R}^n)}\|w\|_{H^1(\mathbb{R}^n)}.$$
We claim that 
\begin{equation}\label{key inequality for rate}
\big\|(-\Delta+1-P_c(D))u_c\big\|_{H^{-1}(\mathbb{R}^n)}=O\left(\frac{1}{c^2}\right).
\end{equation}
Indeed, by the Plancherel theorem and the high Sobolev norm uniform bound (Proposition \ref{Uniform Sobolev norm bounds for the pseudo-relativistic equations}), 
\begin{align*}
\big\|(-\Delta+1-P_c(D))u_c\big\|_{H^{-1}(\mathbb{R}^n)}&=\frac{1}{(2\pi)^{\frac{d}{2}}}\Big\|\frac{1}{\langle\xi\rangle}\Big((|\xi|^2+1)-P_c(\xi)\Big)\widehat{u}_c(\xi)\Big\|_{L_{\xi}^2(\mathbb{R}^n)}\\
&=\frac{1}{(2\pi)^{\frac{d}{2}}}\Big\|\frac{1}{\langle\xi\rangle}O\left(\frac{|\xi|^4}{c^2}\right)\widehat{u}_c(\xi)\Big\|_{L_{\xi}^2(\mathbb{R}^n)}\\
&\leq O\left(\frac{1}{c^2}\right)\frac{1}{(2\pi)^{\frac{d}{2}}}\big\|\langle\xi\rangle^3\widehat{u}_c(\xi)\big\|_{L_{\xi}^2(\mathbb{R}^n)}\\
&=O\left(\frac{1}{c^2}\right)\|u_c\|_{H^3(\mathbb{R}^n)}=O\left(\frac{1}{c^2}\right).
\end{align*}
Thus, by the claim \eqref{key inequality for rate}, we prove that
\begin{equation}\label{Lww estimate1}
\int_{\mathbb{R}^n}(-\Delta+1-P_c(D))u_c\cdot w\ dx= O\left(\frac{1}{c^2}\right)\|w\|_{H^1(\mathbb{R}^n)}.
\end{equation}

For the second integral on the right hand side of \eqref{Lww}, by the H\"older and Sobolev inequalities,
\begin{align*}
&\int_{\mathbb{R}^n}\Big\{\mathcal{N}(u_\infty+w)-\mathcal{N}(u_\infty)-\mathcal{N}'(u_\infty)w\Big\}\cdot w\ dx\\
&\leq\big\|\mathcal{N}(u_\infty+w)-\mathcal{N}(u_\infty)-\mathcal{N}'(u_\infty)w\big\|_{L^{\frac{p}{p-1}}(\mathbb{R}^n)}\|w\|_{L^p(\mathbb{R}^n)}\\
&\lesssim\big\|\mathcal{N}(u_\infty+w)-\mathcal{N}(u_\infty)-\mathcal{N}'(u_\infty)w\big\|_{L^{\frac{p}{p-1}}(\mathbb{R}^n)}\|w\|_{H^1(\mathbb{R}^n)},
\end{align*}
where $p=4$ for the Hartree nonlinearity. We claim that
\begin{equation}\label{w estimate}
\big\|\mathcal{N}(u_\infty+w)-\mathcal{N}(u_\infty)-\mathcal{N}'(u_\infty)w\big\|_{L^{\frac{p}{p-1}}(\mathbb{R}^n)}=O\left(\|w\|_{H^1(\mathbb{R}^n)}^2\right).
\end{equation}
Indeed, for the polynomial nonlinearity, by the H\"older and Sobolev inequalities,
\begin{align*}
&\Big\|(u_\infty+w)^{p-1}-u_\infty^{p-1}-(p-1)u_\infty^{p-2}w\Big\|_{L^{\frac{p}{p-1}}(\mathbb{R}^n)}\\
&=\Big\|\sum_{k=0}^{p-3} \frac{(p-1)!}{k! (p-1-k)!}u_\infty^k w^{p-k-1}\Big\|_{L^{\frac{p}{p-1}}(\mathbb{R}^n)}\\
&\leq\sum_{k=0}^{p-3} \frac{(p-1)!}{k! (p-1-k)!}\|u_\infty\|_{L^p(\mathbb{R}^n)}^k \|w|_{L^p(\mathbb{R}^n)}^{p-k-1}\\
&\lesssim\sum_{k=0}^{p-3} \frac{(p-1)!}{k! (p-1-k)!}\|u_\infty\|_{L^p(\mathbb{R}^n)}^k \|w\|_{H^1(\mathbb{R}^n)}^{p-k-1}\\
&\lesssim\|w\|_{H^1(\mathbb{R}^n)}^2.
\end{align*}
On the other hand, for the Hartree nonlinearity, we have
\begin{equation}\label{eq-4-1}
\begin{split}
&\mathcal{N}(u_\infty+w)-\mathcal{N}(u_\infty)-\mathcal{N}'(u_\infty)w
\\
&=(|x|^{-1} * |u_{\infty} + w|^2 ) (u_{\infty} + w)
\\
&\qquad- (|x|^{-1} * |u_{\infty}|^2)u_{\infty} + 2( |x|^{-1}* (u_{\infty}w)) u_{\infty} + (|x|^{-1} *u_{\infty}^2) w
\\
&=(|x|^{-1} * w^2)u_{\infty} + 2 (|x|^{-1} * (u_{\infty}w)) w + (|x|^{-1} * w^2) w. 
\end{split}
\end{equation}
To estimate this, we use the Hardy-Littlewood-Sobolev inequality with the Sobolev embedding to find the following estimate
\begin{align*}
\| |x|^{-1} (fg) h \|_{L^{4/3} (\mathbb{R}^3)} & \leq \| |x|^{-1} (fg) \|_{L^4 (\mathbb{R}^3)} \| h\|_{L^2 (\mathbb{R}^3)}
\\
&\leq \|fg\|_{L^{12/11}(\mathbb{R}^3)}  \| h\|_{L^2 (\mathbb{R}^3)}
\\
&\leq \|f\|_{L^{24/11}(\mathbb{R}^3)}\|f\|_{L^{24/11}(\mathbb{R}^3)}  \| h\|_{L^2 (\mathbb{R}^3)}
\\
&\leq C\|f\|_{H^1 (\mathbb{R}^3)}\|g\|_{H^1 (\mathbb{R}^3)}\|h\|_{H^1 (\mathbb{R}^3)}.
\end{align*}
By applying this inequality to \eqref{eq-4-1}, we prove the claim for the Hartree nonlinearity, 
\begin{equation*}
\|\mathcal{N}(u_\infty+w)-\mathcal{N}(u_\infty)-\mathcal{N}'(u_\infty)w\|_{L^{4/3}(\mathbb{R}^3)} \leq C \|u_{\infty}\|_{H^1 (\mathbb{R}^3)} \|w\|_{H^1 (\mathbb{R}^3)}^2 +C\|w\|_{H^1 (\mathbb{R}^3)}^3.
\end{equation*} 
Therefore, by the claim \eqref{w estimate}, we obtain
\begin{equation}\label{Lww estimate2}
\int_{\mathbb{R}^n}\Big\{\mathcal{N}(u_\infty+w)-\mathcal{N}(u_\infty)-\mathcal{N}'(u_\infty)w\Big\}\cdot w\ dx=O\left(\|w\|_{H^1(\mathbb{R}^n)}^3\right).
\end{equation}
Going back to \eqref{Lww}, by \eqref{Lww estimate1} and \eqref{Lww estimate2}, we prove that
\begin{equation}\label{Lww'}
\int_{\mathbb{R}^n}\mathcal{L}w\cdot w\ dx=O\left(\frac{1}{c^2}\right)\|w\|_{H^1(\mathbb{R}^n)}+O\left(\|w\|_{H^1(\mathbb{R}^n)}^3\right).
\end{equation}

Now, in order to make use of the non-degeneracy of the linearized operator $\mathcal{L}$ in \eqref{Lww'}, we decompose
$$w=\lambda u_\infty+v\quad\textup{with }\langle v,u_\infty\rangle_{H^1(\mathbb{R}^n)}=0$$
so that $\|w\|_{H^1(\mathbb{R}^n)}^2=\|v\|_{H^1(\mathbb{R}^n)}^2+\lambda^2\|u_\infty\|_{H^1(\mathbb{R}^n)}^2$. Here, $v$ is a radially symmetric function, since $u_\infty$ and $u_c$ are radially symmetric. Thus, by Lemma \ref{lem: non-degeneracy},
\begin{align*}
d\|v\|_{H^1(\mathbb{R}^n)}^2&\leq \int_{\mathbb{R}^n}\mathcal{L}v\cdot v\ dx\\
&=\int_{\mathbb{R}^n}\mathcal{L}w\cdot w\ dx-2\lambda\int_{\mathbb{R}^n}\mathcal{L}u_\infty\cdot v\ dx-\lambda^2\int_{\mathbb{R}^n}\mathcal{L}u_\infty\cdot u_\infty\ dx.
\end{align*}
By \eqref{algebraic property of L} and $\langle u_\infty, v\rangle_{H^1}=0$, we obtain that
\begin{align*}
d\|v\|_{H^1(\mathbb{R}^n)}^2&\leq\int_{\mathbb{R}^n}\mathcal{L}w\cdot w\ dx+2\lambda(p-2)\int_{\mathbb{R}^n}(-\Delta+1)u_\infty\cdot v\ dx\\
&\quad+\lambda^2(p-2)\int_{\mathbb{R}^n}(-\Delta+1)u_\infty\cdot u_\infty\ dx\\
&=\int_{\mathbb{R}^n}\mathcal{L}w\cdot w\ dx+\lambda^2(p-2)\|u_\infty\|_{H^1(\mathbb{R}^n)}^2.
\end{align*}
Therefore, it follows from \eqref{Lww'} and Cauchy-Schwartz that
\begin{align*}
d\|w\|_{H^1(\mathbb{R}^n)}^2&=d\left(\|v\|_{H^1(\mathbb{R}^n)}^2+\lambda^2\|u_\infty\|_{H^1(\mathbb{R}^n)}^2\right)\\
&\leq\int_{\mathbb{R}^n}\mathcal{L}w\cdot w\ dx+\lambda^2(p-1)\|u_\infty\|_{H^1(\mathbb{R}^n)}^2\\
&=O\left(\frac{1}{c^2}\right)\|w\|_{H^1(\mathbb{R}^n)}+O\left(\|w\|_{H^1(\mathbb{R}^n)}^3\right)+O(\lambda^2)\\
&=O\left(\frac{1}{c^4}\right)+\frac{d}{4}\|w\|_{H^1(\mathbb{R}^n)}^2+\frac{}{}O\left(\|w\|_{H^1(\mathbb{R}^n)}^3\right)+O(\lambda^2).
\end{align*}
Since $w\to 0$ in $H^1(\mathbb{R}^n)$ as $c\to\infty$, it implies that
\begin{equation}\label{first error estimate}
\|w\|_{H^1(\mathbb{R}^n)}^2=O\left(\frac{1}{c^4}\right)+O(\lambda^2).
\end{equation}

It remains to estimate $\lambda$. Since $u_c = (1+\lambda)u_\infty+v$, we have
\[
\begin{aligned}
-\Delta u_c +u_c-\mathcal{N}(u_c) &= (1+\lambda)(-\Delta u_\infty+u_\infty-\mathcal{N}(u_\infty))\\
&\quad-\Delta v +v +(1+\lambda)\mathcal{N}(u_\infty)-\mathcal{N}(u_c)
\\
&=-\Delta v +v +(1+\lambda)\mathcal{N}(u_\infty)-\mathcal{N}(u_c).
\end{aligned}
\]
In order to get rid of $-\Delta v+v$, we take the $L^2(\mathbb{R}^n)$-inner product of the equation against $u_\infty$, then
\[
\begin{aligned}
&\int_{\R^n}(-\Delta u_c +u_c-\mathcal{N}(u_c))\cdot u_\infty\,dx 
\\
&= (1+\lambda)\int_{\R^N}\mathcal{N}(u_\infty)\cdot u_\infty\,dx -\int_{\R^n}\mathcal{N}(u_c)\cdot u_\infty\,dx\\
&= \lambda\int_{\R^N}\mathcal{N}(u_\infty)\cdot u_\infty\,dx +\int_{\R^n}(\mathcal{N}(u_\infty)-\mathcal{N}(u_c))\cdot u_\infty\,dx\\
&= \lambda\|u_\infty\|_{H^1(\mathbb{R}^n)}^2 +\int_{\R^n}(\mathcal{N}(u_\infty)-\mathcal{N}(u_c))\cdot u_\infty\,dx,
\end{aligned}
\]
where in the last identity, we used the equation for $u_\infty$. Hence,
\begin{equation}\label{lambda estimate}
\begin{aligned}
 \lambda\|u_\infty\|_{H^1(\mathbb{R}^n)}^2&=\int_{\mathbb{R}^n} (-\Delta u_c+u_c-\mathcal{N}(u_\infty))\cdot u_\infty \ dx\\
&=\int_{\mathbb{R}^n}\Big\{(-\Delta+1-P_c(D))u_c+(P_c(D) u_c-\mathcal{N}(u_\infty))\Big\}\cdot u_\infty \ dx\\
&=\int_{\mathbb{R}^n}\Big\{(-\Delta+1-P_c(D))u_c+(\mathcal{N}(u_c)-\mathcal{N}(u_\infty))\Big\}\cdot u_\infty \ dx\\
&=\int_{\mathbb{R}^n}(-\Delta+1-P_c(D))u_c\cdot u_\infty \ dx+\int_{\mathbb{R}^n} \mathcal{N}'(u_\infty)w\cdot u_\infty \ dx\\
&\quad+\int_{\mathbb{R}^n} \big(\mathcal{N}(u_c)-\mathcal{N}(u_\infty)-\mathcal{N}'(u_\infty)w\big)\cdot u_\infty \ dx,
\end{aligned}
\end{equation}
where in the third identity, we used the equation for $u_c$. For the first integral in \eqref{lambda estimate}, we apply \eqref{key inequality for rate} to get
\begin{align*}
&\int_{\mathbb{R}^n}(-\Delta+1-P_c(D))u_c\cdot u_\infty \ dx\\
&\leq\big\|(-\Delta+1-P_c(D))u_c\big\|_{H^{-1}(\mathbb{R}^n)}\|u_\infty\|_{H^1(\mathbb{R}^n)}\\
&=O\left(\frac{1}{c^2}\right).
\end{align*}
For the second integral, we use \eqref{algebraic property of L} and orthogonality of $u_\infty$ and $v$,  
\begin{align*}
\int_{\mathbb{R}^n} \mathcal{N}'(u_\infty)w\cdot u_\infty \ dx&=\int_{\mathbb{R}^n} \mathcal{N}'(u_\infty)u_\infty\cdot w \ dx=(p-1)\int_{\mathbb{R}^n} \mathcal{N}(u_\infty)\cdot w \ dx\\
&=(p-1)\int_{\mathbb{R}^n} (-\Delta+1)u_\infty\cdot (\lambda u_\infty+v) \ dx\\
&=\lambda(p-1)\|u_\infty\|_{H^1(\mathbb{R}^n)}^2.
\end{align*}
By H\"older inequality and \eqref{w estimate}, the last integral in \eqref{lambda estimate} is bounded by
$$\big\|\mathcal{N}(u_c)-\mathcal{N}(u_\infty)-\mathcal{N}'(u_\infty)w\big\|_{L^{\frac{p}{p-1}}(\mathbb{R}^n)}\|u_\infty\|_{L^p(\mathbb{R}^n)}=O\left(\|w\|_{H^1(\mathbb{R}^n)}^2\right).$$
Collecting all, we prove that
\begin{equation*}
\begin{split}
&\lambda(p-2)\|u_\infty\|_{H^1(\mathbb{R}^n)}^2=O\left(\|w\|_{H^1(\mathbb{R}^n)}^2\right)+ O \left( \frac{1}{c^2}\right) 
\\
&\quad\Rightarrow \lambda=O\left(\|w\|_{H^1(\mathbb{R}^n)}^2\right) + O \left( \frac{1}{c^2}\right).
\end{split}
\end{equation*}
Finally, going back to \eqref{first error estimate}, we conclude that 
$$\|w\|_{H^1(\mathbb{R}^n)}^2=O\left(\frac{1}{c^4}\right)+O\left(\|w\|_{H^1(\mathbb{R}^n)}^4\right)\Rightarrow \|w\|_{H^1(\mathbb{R}^n)}=O\left(\frac{1}{c^2}\right) .$$
\end{proof}

\begin{prop}[Optimality of Proposition \ref{prop: rate of convergence}] There exists a constant $B>0$ such that
\begin{equation}
\|u_c- u_{\infty}\|_{H^1} \geq \frac{B}{c^2}
\end{equation}
for any $c>1$ large enough.
\end{prop}
\begin{proof}
Let $w=u_c- u_{\infty}$ and we write \eqref{eq-s-4} as
\[
\left[\sqrt{-c^2\Delta +\frac{c^4}{4}}-\frac{c^2}{2}\right](u_{\infty}+w) + (u_{\infty}+w) = \mathcal{N}(u_{\infty}+w).
\]
Rearranging this with \eqref{eq-s-3}, we find that 
\begin{equation*}
\begin{split}
&\left[\sqrt{-c^2\Delta +\frac{c^4}{4}}-\frac{c^2}{2}\right](w) + w + \mathcal{N}(u_{\infty}) -\mathcal{N}(u_{\infty}+w)
\\
&\qquad = -\Delta u_{\infty}  -\left[\sqrt{-c^2\Delta +\frac{c^4}{4}}-\frac{c^2}{2}\right](u_{\infty}).
\end{split}
\end{equation*}
Multiplying both sides by $u_{\infty}$ and integrating over $\mathbb{R}^n$, we get
\begin{equation*}
\begin{split}
\mathcal{A}:=&\int_{\mathbb{R}^n} |\nabla u_{\infty}|^2 - u_{\infty}(x) \left[\sqrt{-c^2\Delta +\frac{c^4}{4}}-\frac{c^2}{2}\right](u_{\infty})dx
\\
& =\int_{\mathbb{R}^n} \Biggl(\left[\sqrt{-c^2\Delta +\frac{c^4}{4}}-\frac{c^2}{2}\right](w) + w + \mathcal{N}(u_{\infty}) -\mathcal{N}(u_{\infty}+w)\Biggr) u_{\infty} (x) dx :=\mathcal{B}.
\end{split}
\end{equation*}
We now proceed to obtain upper and lower bounds in the above for the proof. First, we shall find a lower bound of $A$. By the Plancherel theorem,
\begin{equation*}
\mathcal{A} = \int_{\mathbb{R}^n} \widehat{u_{\infty}}(\xi)^2 \left[ \xi^2 - \Bigl( \sqrt{ c^2 \xi^2 + \frac{c^4}{4}} - \frac{c^2}{2}\Bigr) \right] \ d\xi.
\end{equation*}
Notice that for $\xi \in \mathbb{R}^n$,
\begin{equation}\label{eq-4-2}
\xi^2 -\left(\sqrt{ c^2 \xi^2 + \frac{c^4}{4}} - \frac{c^2}{2} \right)= \xi^2- \frac{c^2 \xi^2}{ \sqrt{ c^2 \xi^2 + \frac{c^4}{4}} + \frac{c^2}{2}} \geq 0.
\end{equation}
In addition, a Taylor expansion shows that for $|\xi| \leq 1$, 
\begin{equation*}
\xi^2 - \sqrt{ c^2 \xi^2 + \frac{c^4}{4}} - \frac{c^2}{2} = \frac{2\xi^4}{c^2} + O \left( \frac{1}{c^4}\right) \geq \frac{\xi^4}{c^2}
\end{equation*}
provided $c>1$ is large enough. Therefore, we have
\begin{equation}\label{eq-4-3}
\mathcal{A} \gtrsim \frac{1}{c^2} \int_{|\xi|\leq 1} |\xi|^4\, \widehat{u_{\infty}}(\xi)^2 \, d\xi.
\end{equation}
On the other hand, using the estimates in the proof of Proposition \ref{prop: rate of convergence} and \eqref{eq-4-2}, we deduce
\begin{equation*}
\begin{split}
\mathcal{B} &= \int_{\mathbb{R}^n} \Biggl(  \mathcal{N}(u_{\infty}) -\mathcal{N}(u_{\infty}+w)\Biggr) u_{\infty} (x) dx + \int_{\mathbb{R}^n} w(x)\left[ 1+ \sqrt{-c^2\Delta +\frac{c^4}{4}}-\frac{c^2}{2}\right]u_{\infty}(x) dx
\\
& \leq C\|w\|_{H^1 (\mathbb{R}^n)}\left(\|w\|_{H^1 (\mathbb{R}^n)}+ \|u_{\infty}\|_{H^1 (\mathbb{R}^n)}\right)^{p-1}+ C\|w\|_{H^1 (\mathbb{R}^n)} \|u_{\infty}\|_{H^1 (\mathbb{R}^n)}.
\end{split}
\end{equation*}
Combining this estimate with \eqref{eq-4-3}, we finally get 
\begin{equation*}
 \|w\|_{H^1 (\mathbb{R}^n)} \geq \frac{B}{c^2}
 \end{equation*}
with a constant $D>0$ independent of large $c> 1$. The proof is finished.
\end{proof}

\section{High Sobolev norm convergence from a low Sobolev norm convergence}\label{sec-4}

We prove that the convergence of the pseudo-relativistic ground state in a low Sobolev norm implies that in a high Sobolev norm.

\begin{prop}[High Sobolev norm convergence from a low Sobolev norm convergence]\label{low to high convergence}
\begin{equation}\label{eq: low to high convergence}
\|u_c-u_\infty\|_{H^s(\mathbb{R}^n)}=O(\|u_c-u_\infty\|_{H^{1/2}(\mathbb{R}^n)})+O \left( \frac{1}{c^2}\right),\quad\forall s\geq\tfrac{1}{2}.
\end{equation}
\end{prop}

\begin{proof}
\textbf{Step 1. (Setup of the proof)}
Using the equations for $u_c$ and $u_\infty$, we decompose the difference $u_c-u_\infty$ into the two parts:
\begin{align*}
u_c-u_\infty&= P_c(D)^{-1}\mathcal{N}(u_c)-(1-\Delta)^{-1}\mathcal{N}(u_\infty)\\
&=P_c(D)^{-1}\Big(\mathcal{N}(u_c)-\mathcal{N}(u_\infty)\Big)+\Big(P_c(D)^{-1}-(1-\Delta)^{-1}\Big)\mathcal{N}(u_\infty)\\
&=P_c(D)^{-1}\Big(\mathcal{N}(u_c)-\mathcal{N}(u_\infty)\Big)+\Big(P_c(D)^{-1}-(1-\Delta)^{-1}\Big)(1-\Delta)u_\infty\\
&=P_c(D)^{-1}\Big(\mathcal{N}(u_c)-\mathcal{N}(u_\infty)\Big)+\Big(\frac{1-\Delta}{P_c(D)}-1\Big)u_\infty.
\end{align*}
Note that the first term $P_c(D)^{-1}(\mathcal{N}(u_c)-\mathcal{N}(u_\infty))$, which is the main term, has lower order (in differentiability) than $u_c-u_\infty$. We will use this fact to upgrade the regularity. The remainder $(\frac{1-\Delta}{P_c(D)}-1)u_\infty$ is easier to deal with, since the ground state $u_\infty$ is a well-known nice function.

Let $\{s_k\}_{k=0}^\infty$ be the sequence, given in Proposition \ref{iterated Hartree nonlinear estimates} or in Proposition \ref{iterated integer-power nonlinear estimates}, depending on the nonlinearity such that $s_0=\frac{1}{2}$ and $s_k\to \infty$ as $k\to\infty$.  Then, by the triangle inequality and Lemma \ref{symbol bound},
\begin{align*}
\|u_c-u_\infty\|_{H^{s_{k+1}}}&\leq \Big\|P_c(D)^{-1}\Big(\mathcal{N}(u_c)-\mathcal{N}(u_\infty)\Big)\Big\|_{H^{s_{k+1}}}+\Big\|\Big(\frac{1-\Delta}{P_c(D)}-1\Big)u_\infty\Big\|_{H^{s_{k+1}}}\\
&\lesssim \big\|\mathcal{N}(u_c)-\mathcal{N}(u_\infty)\big\|_{H^{s_{k+1}-1}}+\Big\|\Big(\frac{1-\Delta}{P_c(D)}-1\Big)u_\infty\Big\|_{H^{s_{k+1}}}\\
&=:I+II.
\end{align*}
\textbf{Step 2. (Estimate on $I$)} For the Hartree nonlinearity, we write
$$\mathcal{N}(u_c)-\mathcal{N}(u_\infty)=\Big(\frac{1}{|x|}*(u_c+u_\infty)(u_c-u_\infty)\Big)u_c+\Big(\frac{1}{|x|}*u_\infty^2\Big)(u_c-u_\infty).$$
Then, by Proposition \ref{iterated Hartree nonlinear estimates} and \ref{Uniform Sobolev norm bounds for the pseudo-relativistic equations}, 
$$I\lesssim \Big(\|u_c\|_{H^{s_k}(\mathbb{R}^n)}+\|u_\infty\|_{H^{s_k}(\mathbb{R}^n)}\Big)^2\|u_c-u_\infty\|_{H^{s_k}(\mathbb{R}^n)}\lesssim \|u_c-u_\infty\|_{H^{s_k}(\mathbb{R}^n)}.$$
For the integer-power nonlinearity, we write
\begin{align*}
\mathcal{N}(u_c)-\mathcal{N}(u_\infty)&=u_c^{p-1}-u_\infty^{p-1}
\\
&=\sum_{\ell=0}^{p-2}\frac{(p-1)!}{(\ell+1)!(p-2-\ell)!} u_c^{p-2-\ell} (u_c-u_\infty)^{\ell+1}.
\end{align*}
Then, by Proposition \ref{iterated integer-power nonlinear estimates} and \ref{Uniform Sobolev norm bounds for the pseudo-relativistic equations}, 
\begin{align*}
I\lesssim \Big(\|u_c\|_{H^{s_k}(\mathbb{R}^n)}+\|u_\infty\|_{H^{s_k}(\mathbb{R}^n)}\Big)^{p-2}\|u_c-u_\infty\|_{H^{s_k}(\mathbb{R}^n)}\lesssim \|u_c-u_\infty\|_{H^{s_k}(\mathbb{R}^n)}.
\end{align*}
\textbf{Step 3. (Estimate on $II$)} We observe that if $|\xi|\leq\frac{c}{10}$, by the Taylor's theorem, 
$$\Biggl|\frac{1+|\xi|^2}{P_c(\xi)}-1\Biggr|=\Biggl|\frac{1+|\xi|^2-P_c(\xi)}{P_c(\xi)}\Biggr|\lesssim \frac{\frac{|\xi|^4}{c^2}}{\langle\xi\rangle^2}\leq\frac{|\xi|^2}{c^2},$$
but otherwise, by Lemma \ref{symbol bound},
$$\Biggl|\frac{1+|\xi|^2}{P_c(\xi)}-1\Biggr|\lesssim \langle\xi\rangle\lesssim \frac{|\xi|^3}{c^2}.$$
Thus,
$$II\lesssim\frac{1}{c^2}\|u_\infty\|_{H^{s_k+3}(\mathbb{R}^n)}.$$
\textbf{Step 4. (Conclusion)} In summary, we proved that
$$\|u_c-u_\infty\|_{H^{s_{k+1}}(\mathbb{R}^n)}=O\left(\|u_c-u_\infty\|_{H^{s_k}(\mathbb{R}^n)}\right)+O \left( \frac{1}{c^2}\right).$$
Hence, iterating from $s_0=\frac{1}{2}$, we obtain \eqref{eq: low to high convergence} for all $s=s_k$. Therefore, by interpolation, we conclude that \eqref{eq: low to high convergence} holds for all $s\geq \frac{1}{2}$.
\end{proof}

\appendix

\section{Lorentz spaces}

For $1\leq p,q\leq \infty$, we define the \textit{Lorentz norm}
\begin{equation}
\|f\|_{L^{p,q}(\mathbb{R}^n)}:=\left\{\begin{aligned}
&\Big\{p\int_0^\infty \lambda^{q-1}\big|\{x: |f(x)|\geq\lambda\}\big|^{\frac{q}{p}}d\lambda\Big\}^{\frac{1}{q}}&&\textup{when }1\leq p,q<\infty,\\
&\sup_{\lambda>0}\lambda\big|\{x: |f(x)|\geq\lambda\}\big|^{\frac{1}{p}}&&\textup{when }1\leq p<\infty\textup{ and }q=\infty,\\
&\|f\|_{L^\infty(\mathbb{R}^n)}&&\textup{when }p=q=\infty.
\end{aligned}
\right.
\end{equation}
We define the \textit{Lorentz space} $L^{p,q}(\mathbb{R}^n)$ as the space of all functions having finite $L^{p,q}(\mathbb{R}^n)$-norm. By definition, $L^{p,p}(\mathbb{R}^n)=L^p(\mathbb{R}^n)$ and $L^{p,\infty}(\mathbb{R}^n)$ is the weak $L^p(\mathbb{R}^n)$-space.

We summarize the basic properties of the Lorentz spaces as follows.
\begin{prop} $(i)$ (Inclusion) The embedding $L^{p,q_1}\hookrightarrow L^{p,q_2}$ is continuous if $q_1<q_2$.\\
$(ii)$ (H\"older inequality) If $1\leq p,p_1,p_2<\infty$, $1\leq q,q_1,q_2\leq\infty$, $\frac{1}{p}=\frac{1}{p_1}+\frac{1}{p_2}$ and $\frac{1}{q}=\frac{1}{q_1}+\frac{1}{q_2}$, then
\begin{equation}
\|fg\|_{L^{p,q}(\mathbb{R}^n)}\lesssim \|f\|_{L^{p_1,q_1}(\mathbb{R}^n)}\|g\|_{L^{p_2,q_2}(\mathbb{R}^n)}.
\end{equation}
$(iii)$ (Dual characterization) If $1\leq p, q\leq\infty$, then
\begin{equation}
\|f\|_{L^{p,q}(\mathbb{R}^n)}\sim\sup_{\|g\|_{L^{p',q'}(\mathbb{R}^n)}\leq 1}\int_{\mathbb{R}^n} f(x)\overline{g(x)}dx.
\end{equation}
\end{prop}

The Lorentz spaces are particularly useful to interpolate boundedness of an operator, since its original definition stems from interpolation of the Lebesgue spaces.
\begin{prop}[Marcinkeiwicz interpolation]
Let $1\leq p_0,p_1,q_0,q_1\leq\infty$. Suppose that $T$ is sublinear, i.e.,
\begin{equation}
|T(cf)|\leq |c||Tf|\quad\textup{and }|T(f+g)|\leq |Tf|+|Tg|,
\end{equation} 
and that
\begin{equation}
\|Tf\|_{L^{q_i,\infty}(\mathbb{R}^n)}\lesssim \|f\|_{L^{p_i,1}(\mathbb{R}^n)}\quad\textup{for }i=0,1.
\end{equation}
Then, for any $0<\theta<1$ and $1\leq r<\infty$, we have
\begin{equation}
\|Tf\|_{L^{q_\theta,r}(\mathbb{R}^n)}\lesssim \|f\|_{L^{p_\theta,r}(\mathbb{R}^n)},
\end{equation}
where
$\frac{1}{p_\theta}=\frac{1-\theta}{p_0}+\frac{\theta}{p_1}$ and $\frac{1}{q_\theta}=\frac{1-\theta}{q_0}+\frac{\theta}{q_1}$.
\end{prop}

As a consequence, we obtain the fractional integration inequality, equivalently Sobolev inequality, in the Lorentz spaces.
\begin{prop}[Fractional integration inequality] $(i)$ If $1<p,q<\infty$, $1\leq r<\infty$ and $\frac{1}{q}=\frac{1}{p}-\frac{s}{n}$, then
\begin{equation}
\Big\|\int_{\mathbb{R}^n}\frac{f(y)}{|x-y|^{n-s}}dy\Big\|_{L^{q,r}(\mathbb{R}^n)}\lesssim \|f\|_{L^{p,r}(\mathbb{R}^n)}.
\end{equation}
$(ii)$ (Endpoint cases)
\begin{equation}
\begin{aligned}
\Big\|\int_{\mathbb{R}^n}\frac{f(y)}{|x-y|^{n-s}}dy\Big\|_{L^{\frac{n}{n-s},\infty}(\mathbb{R}^n)}&\lesssim \|f\|_{L^1(\mathbb{R}^n)},\\
\Big\|\int_{\mathbb{R}^n}\frac{f(y)}{|x-y|^{n-s}}dy\Big\|_{L^{\infty}(\mathbb{R}^n)}&\lesssim \|f\|_{L^{\frac{n}{s},1}(\mathbb{R}^n)}.
\end{aligned}
\end{equation}
\end{prop}

\begin{proof}
By the H\"older inequality, we prove the second inequality in the endpoint cases,
\begin{equation}
\begin{aligned}
\Big\|\int_{\mathbb{R}^n}\frac{f(y)}{|x-y|^{n-s}}dy\Big\|_{L^{\infty}(\mathbb{R}^n)}&\lesssim \Big\|\frac{1}{|x-y|^{n-s}}\Big\|_{L^{\frac{n}{n-s},\infty}(\mathbb{R}^n)}\|f\|_{L^{\frac{n}{s},1}(\mathbb{R}^n)}\lesssim\|f\|_{L^{\frac{n}{s},1}(\mathbb{R}^n)}.
\end{aligned}
\end{equation}
Then, the first endpoint case inequality follows from the dual characterization and $(i)$ follows from the interpolation theorem.
\end{proof}

\section{Fractional Leibniz rule}

The fractional Leibniz rule is a useful tool in estimating nonlinear terms.

\begin{thm}[Fractional Leibniz rule \cite{ChWei}]\label{fractional leibniz}
Suppose that $s>0$, $1<r, q_1, p_2<\infty$, $1<p_1, q_2\leq\infty$ and $\frac{1}{r}=\frac{1}{p_1}+\frac{1}{q_1}=\frac{1}{p_2}+\frac{1}{q_2}$. Then,
$$\||\nabla|^s(fg)\|_{L^r(\mathbb{R}^n)}\lesssim \|f\|_{L^{p_1}(\mathbb{R}^n)}\||\nabla|^sg\|_{L^{q_1}(\mathbb{R}^n)}+\||\nabla|^sf\|_{L^{p_2}(\mathbb{R}^n)}\|g\|_{L^{q_2}(\mathbb{R}^n)}.$$
\end{thm}

\begin{rem}
Theorem \ref{fractional leibniz} is a simple extension of the original version in Chirst-Weinstein \cite[Proposition 3.3]{ChWei}, where it is assumed that $n=1$, $0<s<1$ and $1<p_1, q_1, p_2, q_2<\infty$. Indeed, their proof works for all $n\geq 1$ and $s>0$ with no change. Furthermore, one can include the case $p_1=q_2=\infty$ using the boundedness of the maximal function, $\|Mf\|_{L^\infty}\lesssim\|f\|_{L^\infty}$, in the very last step of the proof. 
\end{rem}

\end{document}